\documentclass{aupl}
\usepackage{
amsfonts,
latexsym,
amssymb,
verbatim,
mathrsfs,
}

\numberwithin{equation}{section}

\newcommand{\labbel}{\label}

\newtheorem{theorem}{Theorem}[section]
\newtheorem{lemma}[theorem]{Lemma}

\newtheorem{proposition}[theorem]{Proposition} 
 
\newtheorem{corollary}[theorem]{Corollary}

\newtheorem*{claim*}{Claim}

\newtheorem*{theorem*}{Theorem}
\newtheorem*{proposition*}{Proposition}
\newtheorem*{corollary*}{Corollary}
\newtheorem*{lemma*}{Lemma}
\newtheorem*{scholion*}{Scholion}

\theoremstyle{definition}

\theoremstyle{remark}
\newtheorem{remark}[theorem]{Remark}
\newtheorem{remarks}[theorem]{Remarks}
\newtheorem*{remark*}{Remark}
\newtheorem*{remarks*}{Remarks}
 
\newtheorem{example}[theorem]{Example}

\newtheorem{convention}[theorem]{Convention}

\newtheorem*{observation*}{Observation}

\DeclareMathOperator {\con}{\mathbf {Con}}

 \allowdisplaybreaks[1]

\begin{document}
 
\title
{Relation identities in  3-distributive varieties}
 
\author{Paolo Lipparini} 
\urladdr{http://www.mat.uniroma2.it/~ lipparin}
\address{Dipartimento Identitario di Matematica\\Viale della  Ricerca
 Scientifica\\Universit\`a di Roma ``Tor Vergata'' 
\\I-00133 ROME ITALY}


\thanks{Work performed under the auspices of G.N.S.A.G.A. Work 
 supported by PRIN 2012 ``Logica, Modelli e Insiemi''.
The author acknowledges the MIUR Department Project awarded to the
Department of Mathematics, University of Rome Tor Vergata, CUP
E83C18000100006.}

\begin{abstract}
Let $\alpha$, $\beta$, $\gamma, \dots$ 
$\Theta$, $\Psi, \dots$  $R$, $S$, $T, \dots$
be variables for, respectively, congruences, 
 tolerances and reflexive admissible relations.
Let  juxtaposition denote intersection.
We show that if
the identity
\begin{equation*}
\alpha( \beta \circ \Theta ) \subseteq \alpha \beta \circ 
\alpha \Theta  \circ \alpha \beta  
 \end{equation*}
holds  in
 a variety $\mathcal {V}$,
then $\mathcal {V}$ has a majority term, equivalently,
$\mathcal {V}$ satisfies 
$ \alpha ( \beta \circ \gamma ) \subseteq \alpha \beta \circ \alpha \gamma $.  
The result is unexpected, since in the displayed identity 
we have one more factor on the right  and,
moreover, if we let $\Theta$ be a congruence,
we get a condition equivalent to $3$-distributivity,
which is well-known to be strictly weaker than the existence
of a majority term.

The above result is optimal in many senses;
for example, we show that slight variations on the displayed identity,
such as 
$ R  (S \circ \gamma  )  \subseteq 
   R S \circ  R   \gamma  \circ R   S$ or 
$R( S \circ T  )  \subseteq R S \circ RT \circ  RT \circ RS$
hold in every  $3$-distributive variety,
hence do not imply  the existence
of a majority term.
Similar identities are valid 
even in varieties with $2$  Gumm terms,
with no distributivity assumption. 
We also discuss relation identities  in
$n$-permutable varieties
and present a few remarks about
implication algebras.
\end{abstract} 

\maketitle

\section{Introduction} \labbel{intro} 
\subsection*{Congruence identities and relation identities} \labbel{rci} 
By a classical and surprising result by Nation \cite{N} in the 1970's,
 there are  non-equivalent lattice identities
 which are equivalent when considered as  identities 
satisfied in all congruence lattices of algebras in some variety.
Many results of this kind followed, for example, 
Freese and J{\'o}nsson  \cite{FJ}  proved that  modularity 
is equivalent to the  Arguesian identity for congruence lattices
in varieties. See also  Day and Freese \cite{DF}.
The survey \cite{cd} by J{\'o}nsson  is  an excellent introduction to 
earlier results on the subject.
The field is still very active today;
more recent results 
and further references 
can be found, among others, 
in Cz\'edli,  Horv\'ath and
 Lipparini \cite{CHL},
 Freese and  McKenzie \cite{FMK},  Gumm \cite{G}, 
Hobby and McKenzie \cite{HMK},  
Kearnes and  Kiss \cite{kk} and Tschantz \cite{T}. 

The  research in the present note  is motivated by two facts
related to the study of congruence identities.
First,
it almost invariably happens that
 tolerances, and sometimes even reflexive and admissible relations,
are irreplaceable tools in proving  results about
congruences. 
This is 
already evident in the proofs of the 
famous and well-known   characterizations by
J{\'o}nsson \cite{J}  and Day \cite{D}
of, respectively, congruence distributive and congruence modular varieties.
The relevance of tolerances and admissible relations
in the study of congruence identities
appears clearly 
in  \cite{CHL,G,T}, just to mention some examples.

The second aspect that plays a role in our motivations is that,
almost at the same time of  Nation's discovery mentioned at the beginning,
non trivial ``relation identities'' have been discovered.
By a \emph{relation identity} we mean an identity satisfied by
reflexive and admissible relations on some algebra, where the operations
considered are usually intersection and relational composition, possibly
also converse and transitive closure.
Non trivial is intended in a sense similar to that  of Nation's results;
namely, we call an implication between two relation identities 
\emph{non trivial} if it holds when considered
 for all algebras in a variety, but it does not necessarily hold for
single algebras.
 Werner \cite{W} showed that a variety 
is congruence permutable if and only if 
the relation identity $R \circ R = R$ holds in $\mathcal {V}$.  
The result is also due independently to Hutchinson \cite{H}.
It is easy to see that the result is non trivial in the sense specified above;
see Remark \ref{ror}(b) below. 
A similar relational characterization of congruence $n$-permutable varieties is
mentioned in  Hagemann an  Mitschke
\cite{HM}. See Proposition  \ref{nperm} below for some generalizations. 
Further examples, details, comments, applications and references
about the use of tolerances or reflexive  admissible relations, and about
relation identities can be found in
\cite{CHL,gm,G,kk,contol,ricmc,uar,baker,ntcm,T}. 
Because of the above comments
 we believe that the study of relation identities is interesting, as a 
non trivial generalization of the study of congruence 
identities.

\subsection*{Relation identities in congruence distributive varieties} \labbel{spc} 
In \cite{ntcm}  we used relation identities in order to approach the problem
of the relationships between the numbers of Gumm and of Day
terms in a congruence modular variety.
While the problem is still largely unsolved,
the partial results confirm the usefulness of the approach.
Though our main aim has
been the study of relation identities
satisfied in  modular varieties,
we encountered delicate issues already
in the relatively well-behaved case
of $4$-distributive varieties \cite{baker}.

Here we show that the problem of the 
satisfaction of relation identities
  is not trivial even for
identities related to
 $3$-distributivity.
A detailed study of $3$-distributive 
varieties in a different direction 
 appears in 
Kiss and Valeriote \cite{KV}.
A  variety $\mathcal {V}$ is \emph{$n$-distributive}
if $\mathcal {V}$ 
satisfies the congruence identity
$\alpha( \beta \circ \gamma ) \subseteq 
\alpha \beta \circ \alpha \beta \circ \alpha \gamma \dots$ ($n$ factors),
 where ``$n$ factors''  means 
$n-1$ occurrences of  $\circ$ on the right-hand side.
In general, we say that 
some identity 
\emph{holds in a variety} $\mathcal {V}$
if the identity holds in the set of reflexive and admissible relations on
every algebra in $\mathcal {V}$.
We shall adopt the conventions introduced in the first 
paragraph of the abstract, in particular, 
juxtaposition denotes intersection and $R$, $S$\dots \ 
are interpreted as reflexive and admissible relations. 
All the binary relations considered in this note are
assumed to be reflexive, hence we shall sometimes simply say 
\emph{admissible} in place of reflexive and admissible.
 
The definition  of 
$n$-distributivity as given above  is equivalent
to the classical notion of being
$\Delta_n$, as  introduced in 
J{\'o}nsson \cite{J}. Among other, J{\'o}nsson \cite{J}
 showed that a variety is
congruence distributive if and only if it is $\Delta_n$,   for some $n$. 
As remarked in \cite{jds}, a recent result by
Kazda, Kozik,   McKenzie and Moore \cite{adjt}  
can be used to strengthen  
J{\'o}nsson's Theorem to the effect that
a variety $\mathcal V$ is congruence distributive 
if and only if there is some $k$ such that  $\mathcal V$ satisfies the relation identity
\begin{equation}\labbel{jr}
 \alpha (S \circ T) \subseteq \alpha S \circ  \alpha T \circ \alpha S \dots
\text{ \ \ ($k$ factors).} 
   \end{equation}    
Hence it is interesting to study 
relation identities satisfied by congruence distributive varieties;
in particular, to evaluate the best 
possible value of $k$ for which \eqref{jr} holds
in any given variety.  

\subsection*{Summary of the main results} \labbel{summ} 
It is easy to see that 
if $\mathcal {V}$ has a majority term
(and is not a trivial variety) then 
the best possible value of $k$ for $\mathcal {V}$ in 
\eqref{jr} above is $2$.
In \cite{baker} we provided examples of 
$4$-distributive varieties 
for which the best possible value of $k$ is $4$.   
The main results of the present note 
assert that there is no variety 
 for which the best possible value of $k$
is  $3$, even if we allow $T$ 
to be a tolerance; moreover, the best value of $k$
is $4$ for every $3$-distributive not  $2$-distributive variety.   
Combining the present results with 
\cite{baker}, we get that the best value of $k$ in  
\eqref{jr} does not determine 
the distributivity level of a variety.

We do not know whether there is a variety 
for which  $k=5$
is the best value.
More generally, we know little about the best possible values of 
$k$ for arbitrary congruence distributive 
varieties, though some bounds are provided
in \cite{jds,baker} for, say,  varieties with $h$
directed J{\'o}nsson terms and for varieties with an
$n$-ary near-unanimity term.
Bounds are provided in Section \ref{np} below  for
congruence distributive $n$-permutable varieties. 

The identities we consider in this note are 
extremely sensible to minimal variations.
For example, if $k=3$ and we let
 $T$ be a congruence in \eqref{jr},
we get a condition equivalent to  $3$-distributivity.
See identity \eqref{3c} 
in Theorem \ref{3dist}. It is well-known that there are 
$3$-distributive not $2$-distributive varieties; in particular, any such variety
fails to satisfy the identity 
\begin{equation}\labbel{da}    
\alpha( \beta \circ \Theta ) \subseteq \alpha \beta \circ 
\alpha \Theta  \circ \alpha \beta  
   \end{equation} 
displayed in the abstract, but does satisfy
\begin{equation}\labbel{daa}     
\alpha ( \Psi \circ  \gamma ) \subseteq 
\alpha \Psi \circ \alpha \gamma  \circ \alpha \Psi.
  \end{equation}
The symmetry between \eqref{da} and  \eqref{daa} is only apparent.
In \eqref{daa}  the relation assumed to be a congruence is placed in the middle
of the right-hand side, while in \eqref{da} 
 it appears two times on the edges.
  
Further relation identities satisfied by $3$-distributive varieties 
are presented in Section \ref{3d}. There we show that, in many cases,
we do not even need the equation $j_2(z,y,z) =z$
from the set of J{\'o}nsson's equations characterizing $3$-distributivity.
Namely, the results apply to congruence modular varieties 
with $2$ Gumm terms, where in the numbering we are not counting the
trivial projection included in the original presentation by Gumm of his set of terms.  

Congruence distributive and congruence modular varieties
which are further assumed to be $n$-permutable are discussed in
Section \ref{np}.   
Finally, in Section \ref{ianu}
we show directly by an example that the displayed identity
in the abstract fails in the variety of implication algebras.  
We also show that most properties of implication algebras 
are preserved if we add a $4$-ary near-unanimity term.

\section{A tolerance identity implying $2$-distributivity} \labbel{tolimp}

\begin{theorem} \labbel{tol2}
If  $\mathcal {V}$ is a variety and the  identity
\begin{equation}\labbel{id} 
\alpha( \beta \circ \Theta ) \subseteq \alpha \beta \circ  \alpha \Theta  \circ \alpha \beta 
 \end{equation}
holds in every algebra 
$\mathbf A $ in $  \mathcal {V}$,
 for all congruences $\alpha$, $\beta$ and tolerance $\Theta$
on $\mathbf A$,
then $\mathcal {V}$ has a majority term.
In fact, it is enough to assume that 
\eqref{id} holds in $\mathbf F _{ \mathcal V } ( 3 ) $,
the free algebra in $\mathcal {V}$ generated by $3$ elements.  
 \end{theorem} 

Theorem \ref{tol2}
shall be proved after  a lemma,
whose proof  is standard,
but not completely usual, since it involves
a tolerance, not only congruences.
Before stating the lemma we recall
the \emph{J{\'o}nsson terms for $3$-distributivity}.
These terms will appear only marginally in the proof of 
the lemma  and will assume a prominent role in Section \ref{3d}.
According to \cite{J}, a variety   $\mathcal {V}$ is
$\Delta_3$
if and only if $\mathcal {V}$ has terms $j_1$ and 
$j_2$ satisfying the following set
of equations:
\begin{equation} \tag{J}
\begin{aligned}\labbel{ggg}
\text{\rm (J$_L$)} \  x & = j_1(x,x,z), 
& \text{\rm (J$_C$)} \   j_1(x,z,z)  & = j_2(x,z,z),   
& \text{\rm (J$_R$)} \    j_2(x,x,z) &  = z,
\\
\text{\rm (J$_1$)}\  x &= j_1(x,y,x),  
&&
&\text{\rm (J$_2$)} \  j_2(z,y,z) & = z.
\end{aligned} 
\end{equation}   

It is implicit in \cite{J},
and only a minor part of it,
 that a variety $\mathcal {V}$ 
is $\Delta_3$ 
if and only if $\mathcal {V}$ obeys the definition 
of $3$-distributivity that we have given in the introduction.  

In the equations \eqref{ggg}
 we could have done with just two 
variables, 
 but we shall not need this observation,
which nevertheless is important and relevant in different contexts.
We believe that here maintaining
three variables looks intuitively clearer;
the same applies to the
 equations in \eqref{eqw} below.

  \begin{remark} \labbel{asym}   
Notice the asymmetry between 
$j_1$ and  $j_2$ in the equations  \eqref{ggg}. 
We get different identities if we reverse 
both the order of the terms and of the variables.
A probably better way to appreciate the asymmetry goes as follows.
Recall that a \emph{majority term} 
is a term  $ j_1$ satisfying the equations
 (J$_L$), 
(J$_1$), as well as $  j_1(x,z,z) = z$.
In \eqref{ggg} $ j_1$  is required to 
satisfy two thirds of the majority condition.
On the other hand, 
a \emph{Maltsev term  for congruence permutability}
is a term $j_2$ satisfying  
 (J$_R$), as well as $x= j_2(x,z,z)$.  
In \eqref{ggg} $ j_2$  is required to 
satisfy one third of the majority condition
and one half of the permutability condition.
The terms
$ j_1$ and $ j_2$ are tied by 
equation (J$_C$), 
but (J$_C$) is not symmetric, either.
Had we required
$ j_1(x,x,z)  = j_2(x,x,z)$ in place of (J$_C$),
we would have obtained $x=z$, by 
(J$_L$) and (J$_R$).  
In spite of the apparent similarity,
$j_1$ and  $j_2$ behave in a quite different way!
Let us remark that the above observation 
is only part of  much more
general considerations, see Remark 
4.2 in \cite{ntcm}.
 \end{remark}

Now 
we state and prove
the lemma we need.
In some of the equations below we shall use a semicolon in place of 
a comma in order to improve readability.

\begin{lemma} \labbel{lemid}
If $\mathbf F _{ \mathcal V } ( 3 ) $   satisfies the identity 
\eqref{id}, then $\mathcal {V}$ has a
$5$-ary term  
$w$ such that the following equations hold in $\mathcal {V}$:
\begin{equation}     
\begin{aligned}\labbel{eqw}
\text{\rm (A)} \ \ x &= w(x,x,z;x,z)    \qquad    
& \text{\rm (B)} \ \      w(x,x, z;z,x) & = z
\\
\text{\rm (C)} \ \
x &= w(x,y,x; y,x)         
 & \text{\rm (D)} \ \   w(x,y,x;x,y )&= x .
\end{aligned} 
 \end{equation}
 \end{lemma} 

\begin{proof}
Let $x$, $y$ and  $z$
be the generators of  $\mathbf F _{ \mathcal V } ( 3 ) $
and let $\alpha$ and $\beta$ be the congruences generated by, respectively,
the pairs $(x,z)$ and  $(x,y)$. Let $\Theta$
be the smallest tolerance containing $(y,z)$, thus
$(x,z) \in \alpha ( \beta \circ \Theta )$, as witnessed by the element 
$y$.   
By assumption, 
$(x,z) \in  \alpha \beta \circ \alpha \Theta  \circ \alpha \beta$, hence 
$ x \mathrel {\alpha  \beta } j_1 \mathrel { \alpha \Theta  } j_2
\mathrel { \alpha \beta  }  z$, for certain elements 
$j_1, j_2 \in F _{ \mathcal V } ( 3 ) $.
Notice that the condition $ j_1 \mathrel { \alpha } j_2$ 
is redundant, since it follows from the assumption that $\alpha$ 
is a congruence and  from
$ j_1 \mathrel { \alpha  } x  \mathrel \alpha  z \mathrel \alpha  j_2$. 
In fact, we shall not use $ j_1 \mathrel { \alpha } j_2$
explicitly.

Since we are working in the free algebra generated by 
 $x$, $y$ and  $z$, we can think of 
$j_1 $ and $ j_2 $ as ternary terms.
Classically, the $ \alpha $-relations 
imply that 
$j_1$ and  $j_2$ satisfy 
 the equations
(J$_1$) and (J$_2$)
from  \eqref{ggg}, 
while the 
$\beta$-relations entail
(J$_L$) and (J$_R$).
Were $\Theta$ assumed to be a congruence, we 
had also (J$_C$), 
thus getting the  J{\'o}nsson's 
condition for $3$-distributivity.
Since $\Theta$ is only assumed to be a tolerance, we have to proceed in a
different fashion. 
It is easy to check that
\begin{equation}\labbel{uu}      
\Theta = \{ (w(x,y,z; y,z), w(x,y,z; z,y))  \mid w \text{ a $5$-ary term of  } 
\mathcal {V} \}.
  \end{equation}
 Indeed, the  
relation defined by the condition on the right 
in \eqref{uu} 
is obviously reflexive, admissible,
symmetrical  and contains  $(y,z)$.
On the other hand, any  reflexive, admissible and
symmetrical  relation containing  $(y,z)$
has to contain all the pairs appearing on the right in \eqref{uu}.
Hence equality follows.

Since $j_1(x,y,z) \mathrel \Theta  j_2(x,y,z)$ by
construction, then
by \eqref{uu} there exists some $5$-ary term   
$w$ such that  
\begin{equation}\labbel{ww}       
j_1(x,y,z) = w(x,y,z; y,z) \qquad  \text{ and } \qquad     w(x,y,z; z,y) = j_2(x,y,z).
\end{equation}
Since the equations in \eqref{ww}  involve only three variables
and hold in $\mathbf F _{ \mathcal V } ( 3 ) $, these equations  hold 
throughout $\mathcal {V}$.
Notice that in \eqref{ww}
we actually need three variables;
two variables are not enough. 
Substituting 
the equations \eqref{ww} in 
(J$_L$), 
(J$_R$), (J$_1$) and (J$_2$),
we get respectively   
the equations (A), (B), (C) and (D) in \eqref{eqw}. 
 \end{proof}    

\begin{remark} \labbel{23}     
Of course, the proof of 
Lemma \ref{lemid}  provides an ``if and only if'' 
condition, since we can retrieve $j_1$ and  $j_2$   
from $w$, using the equations \eqref{ww};
 then the classical homomorphism argument shows 
that these three terms witness that 
the identity \eqref{id} holds throughout $\mathcal {V}$.   
However, we do not need this argument, since
the conclusion of Theorem \ref{tol2} is much stronger;
 indeed,  the existence of a majority term implies the tolerance identity
$\alpha( \beta \circ \Theta ) \subseteq \alpha \beta \circ  \alpha \Theta $, actually,
the relation identity
$R (S \circ T) \subseteq RS \circ RT$.
In other words,
 we get the quite surprising result
that, globally, that is, within a variety,
 the locally weaker tolerance identity
$\alpha( \beta \circ \Theta ) \subseteq 
\alpha \beta \circ  \alpha \Theta  \circ \alpha \beta $
implies  (and hence is equivalent to)
the  much stronger
relation identity
$R (S \circ T) \subseteq RS \circ RT$.
Notice that the latter identity is stronger in two senses:
first, we have $2$ factors instead of $3$ factors on the right;
moreover, it is more general, since it deals with reflexive and admissible relations,
rather than with tolerances or congruences.   

But we are getting  too far  ahead!
A proof of  Theorem \ref{tol2} is necessary in order
 to fully justify the above comment.
\end{remark}

\begin{proof}[Proof of Theorem \ref{tol2}]
Suppose that $\mathcal {V}$ satisfies 
the identity \eqref{id}. By Lemma \ref{lemid},
we have a $5$-ary term $w$
satisfying the equations \eqref{eqw}.
Define
\begin{equation}\labbel{mw}
m(x,y,z) =
w(
w(x,y,z;y,z),
w(x,y,z;z,y),
y; z,
w(x,z,y;y,z)).
   \end{equation}        
 Using \eqref{mw} and \eqref{eqw} repeatedly, we get
\begin{align*}
m(x,x,z) &=
w(
w(x,x,z;x,z),
w(x,x,z;z,x),
x; z,
w(x,z,x;x,z)) 
\\
& = ^{\text{\rm (A)(B)(D)}}
w(x,z,x;z,x) = ^{\text{\rm (C)}} x ,
\\
m(x,y,x) &=
w(
w(x,y,x;y,x),
w(x,y,x;x,y),
y; x,
w(x,x,y;y,x))
\\
& = ^{\text{\rm (C)(D)(B)}}
w(x,x,y;x,y) = ^{\text{\rm (A)}}x ,
\\
m(x,z,z) &=
w(
w(x,z,z;z,z),
w(x,z,z;z,z),
z; z,
w(x,z,z;z,z)) = ^{\text{\rm (B)}}z,
\end{align*}     
where the superscripts indicate
the specified equations of \eqref{eqw}
that we have used, in the respective order.
Hence $m$ 
is a majority term and the theorem 
is proved.
\end{proof}

\section{Relation identities valid in  3-distributive varieties} \labbel{3d}

Since there are $3$-distributive varieties 
without a majority term
(see, e.~g.,  Section \ref{ianu} below),
we have that  $3$-distributive varieties 
do not necessarily satisfy 
identity \eqref{id}.
On the other hand in the present section
we show that 
 $3$-distributive varieties satisfy 
identities which are slightly weaker than 
 \eqref{id}, but very similar to it.
For example, 
 $3$-distributive varieties do satisfy 
 $R  (S \circ \gamma  )  \subseteq 
   R S \circ  R   \gamma  \circ R   S$
and
$ R(S \circ T )  \subseteq 
RS \circ R T \circ R T \circ R S$. 
Notice that if we take $R$, $S$ and $T$
to be congruences in the above identities, we get back
$3$-distributivity, hence the results are  optimal. 

As  mentioned in the
introduction, the results form  Kazda,  Kozik,  McKenzie and Moore \cite{adjt}
and the observations in \cite{jds} show that, for every 
$n$, there is some (possibly quite large) $k$ such that 
$\alpha(S \circ T) \subseteq \alpha S \circ \alpha T \circ \alpha S \circ \dots 
\text{ ($k$ factors)} $    holds in every $n$-distributive variety.
We show that in the case $n=3$ the best possible value of 
$k$ is $4$, hence  not particularly large; moreover, 
we can equivalently consider a reflexive and admissible relation 
$R$ in place of the congruence $\alpha$.

\begin{theorem} \labbel{3dist} 
If $\mathcal {V}$ is a $3$-distributive variety, then
$\mathcal {V}$ satisfies 
\begin{align}
\labbel{3b}  
  R  (S \circ T ) & \subseteq 
 R  S \circ  R  T \circ  R  T \circ  R  S ,
\\
\labbel{3a}     
  R  (S \circ T ) & \subseteq 
 R  S \circ  R  T \circ  R  S \circ  R  T ,  \text{ and }  
\\
\labbel{3c}  
 R  (S \circ \gamma  ) & \subseteq 
   R S \circ  R   \gamma  \circ R   S. 
 \end{align}
\end{theorem}

Recall the equations displayed in  \eqref{ggg}
at the beginning of the previous section. 
 Before proving Theorem \ref{3dist} and for later use,
we shall state in a lemma those parts of the proof 
which do not need the equation (J$_2$).
We first  give a name to some elements which will be frequently used.

\begin{convention} \labbel{conv} 
If $a$, $b$ and  $c$ are elements of some algebra
and $j_1$ and  $j_2$ are ternary terms, we shall denote by
$e$, $f$, $g$, $g^{\bullet}$ the following elements.
\begin{align*}
e &= j_1(j_1(a,b,c),b,c),
\\
f &=   j_1(j_1(a,c,c),c,c) 
= ^{\text{(J$_C$)}} j_2(j_2(a,c,c),c,c),
\\
g &=  j_2(j_2(a,c,c),j_2(b,c,c),c),
\\
  g^{\bullet} &=  j_2(j_2(a,c,c),j_2(a,b,c),c).
 \end{align*}

Since the elements 
$a$, $b$ and  $c$ 
and the terms $j_1$ and  $j_2$ 
shall be kept fixed throughout this section,
we do not indicate explicitly the dependence.
\end{convention}

\begin{lemma} \labbel{lem3}
Under the above Convention, if 
$a \mathrel R c$,
$a \mathrel {S } b \mathrel { T} c $
and $j_1$ and  $j_2$ satisfy
(J$_L$), (J$_C$), (J$_R$)  and (J$_1$), 
then the following relations hold.
\begin{align*}
e \mathrel { R} c,
 && f \mathrel { R} c,  
 && a \mathrel {R  S} e \mathrel { R  T}
 f \mathrel {  T} g \mathrel {  S} c,
&& f \mathrel S g^{\bullet} \mathrel T c.
  \end{align*}      
 \end{lemma}

 \begin{proof} 
The proof is given by the following computations in which,
for clarity, we underline the elements which are moved by the
appropriate relations and, as usual by now, 
we indicate  as superscripts the identities we are using.
\begin{align*}
e= j_1(j_1( \underline{a}, b, c), b,c)
 & \mathrel { R }
 j_1(j_1( \underline{c}, b, c), b,c) = ^{\text{(J$_1$)}} c,
\\
f = j_2(j_2( \underline{ a}, c,  c), c, c) 
& \mathrel R 
 j_2(j_2( \underline{ c}, c,  c), c, c) 
=  ^{\text{(J$_R$)}} c,
\\
 a= ^{\text{(J$_1$)}} j_1(j_1(a, b, \underline{a}), b, \underline{a})
 & \mathrel { R} 
j_1(j_1(a, b, \underline{c}), b,\underline{c}) = e ,
\\
a= ^{\text{(J$_L$)}} j_1(j_1(a, \underline{a},c), \underline{a},c)
 & \mathrel { S}
 j_1(j_1(a, \underline{b},c), \underline{b},c) = e ,
\\
 e= ^{\text{(J$_R$)}}  j_2(j_2( a, \underline{a},  c), c, \underline{e}) 
& \mathrel { R } 
j_2(j_2( a, \underline{c},  c), c, \underline{c}) = f.
\\
e= j_1(j_1(a, \underline{b},c), \underline{b},c) 
& \mathrel {T }
  j_1(j_1(a, \underline{c},c), \underline{c},c) =f .
\\ 
f 
 =  ^{\text{(J$_R$)}}
 j_2(j_2( a, c,  c), j_2( b, \underline{b},  c) , c)
& \mathrel { T }
 j_2(j_2( a, c,  c), j_2( b, \underline{c},  c) , c)  = g,
\\
g = j_2 ( j_2 ( \underline{a},c,c), j_2 (b, c ,c),  c  ) 
& \mathrel { S } 
 j_2 ( j_2 ( \underline{b},c,c), j_2 (b, c ,c),  c  ) 
=^{\text{(J$_R$)}} c,
\\
f  
 =  ^{\text{(J$_R$)}}
 j_2(j_2( a, c,  c), j_2( a, \underline{a},  c) , c)
& \mathrel { S }
  j_2(j_2( a, c,  c), j_2( a, \underline{b},  c) , c) = g^{\bullet},
\\
 g^{\bullet} = j_2 ( j_2 ( a,c,c), j_2 (a, \underline{ b} ,c),  c  ) 
& \mathrel { T } 
 j_2 ( j_2 ( a,c,c), j_2 (a, \underline{ c} ,c),  c  ) 
= ^{\text{(J$_R$)}} c.
\qedhere
\end{align*}    
\end{proof}

  \begin{proof}[Proof of Theorem \ref{3dist}]  
Suppose that  $(a,c) \in   R  (S \circ T )$, 
hence $a \mathrel  R   c$ and
$a \mathrel {S } b \mathrel {T } c  $, for some $b$.
By $3$-distributivity 
and \cite{J},
we have J{\'o}nsson's terms
$j_1$ and  $j_2$ 
satisfying the displayed set of equations  \eqref{ggg}
at the beginning of the previous section. 
Because of Lemma \ref{lem3},
in order to prove 
\eqref{3b} and \eqref{3a} 
it is enough to prove  
$ f \mathrel {  R  } g \mathrel { R  } c$
and   
$ f \mathrel {  R  } g^{\bullet} \mathrel { R  } c$.
We first prove 
$f \mathrel { R} j_2( a, c,  c)$
and we shall also use  
$f \mathrel { R} c$,
as proved in Lemma \ref{lem3}.
Here are the computations. 
\begin{align*} 
f 
 =  ^{\text{(J$_R$)}}
 j_2(j_2( a, c,  c), c, j_2( a, \underline{a},  c))
& \mathrel R 
j_2(j_2( a, c,  c), c, j_2( a, \underline{c},  c)) 
 =  ^{\text{(J$_2$)}}
j_2( a, c,  c), 
\\
f =  ^{\text{(J$_2$)}} j_2 ( \underline{ f}, j_2 (b,c,c), \underline{ f}  )
& \mathrel { R } 
j_2 ( \underline{ j_2 (a,c,c)}, j_2 (b,c,c), \underline{ c}  ) = g,
\\
 g = j_2 ( j_2 ( \underline{a},c,c), j_2 (b, c ,c),  c  ) 
& \mathrel { R } 
 j_2 ( j_2 ( \underline{c},c,c), j_2 (b, c ,c),  c  ) 
=^{\text{(J$_2$)}} c.
\\
f  =  ^{\text{(J$_2$)}} j_2 ( \underline{ f}, j_2 (a,b,c), \underline{ f}  )
& \mathrel { R } 
j_2 ( \underline{ j_2 (a,c,c)}, j_2 (a,b,c), \underline{ c}  ) = g^{\bullet},
\\
g^{\bullet} = j_2 ( j_2 ( \underline{a},c,c), j_2 (a, b ,c),  c  ) 
& \mathrel { R } 
 j_2 ( j_2 ( \underline{c},c,c), j_2 (a, b ,c),  c  ) 
=^{\text{(J$_2$)}} c.
 \end{align*}     
 
Were $R \gamma $ a transitive relation,
\eqref{3c} would follow from \eqref{3b}
with $T = \gamma $.
In the general case, since $\gamma$ is supposed to be a congruence, hence
transitive, we have $e \mathrel { \gamma } g $
from Lemma \ref{lem3}.
In order to prove \eqref{3c}
it is then enough to show 
 $e \mathrel { R } g $.
We have already proved that
$e \mathrel { R } c$, thus
\begin{align*}
e=^{\text{(J$_R$)}} j_2( a, \underline{a}, \underline{ e})
 & \mathrel { R }
 j_2(a, \underline{ c}, \underline{ c}), \text{ hence } 
\\
e=^{\text{(J$_2$)}} j_2( \underline{e}, j_2(b,c,c), \underline{ e})
 & \mathrel { R }
j_2( \underline{j_2(a,c,c)}, j_2(b,c,c), \underline{ c}) =g.
 \end{align*}    

In the special case when
$R $ is a congruence, a 
direct simpler proof of \eqref{3c}
is possible. 
Take $R = \alpha $ and  notice that
if $(a,c) \in   \alpha  (S \circ  \gamma )$, 
as witnessed by $b$, then 
$ a \mathrel { \alpha  S} j_1(a,b,c) \mathrel {   \alpha   \gamma } j_2(a,b,c) 
\mathrel { \alpha  S} c  $. For example,  
$j_1(a,b,c) \mathrel {  \gamma } j_1(a,c,c) =   j_2(a,c,c)
\mathrel \gamma  j_2(a,b,c) $. The rest is more direct and easier.
Cf.\ also \cite[Remark 17]{contol}. 
\end{proof}

\begin{remark} \labbel{uarr}
As we mentioned, if we take 
$R$, $S$ and $T$ congruences in 
the identities \eqref{3b} and  \eqref{3c},
we get back $3$-distributivity.
On the other hand, 
if we take 
$R$, $S$ and $T$ congruences in
\eqref{3a}, we get only $4$-distributivity.
This suggests that perhaps,
for general $n$-distributive varieties, 
 one gets cleaner
results when considering identities of the form
\begin{equation*}\labbel{jrb}
 \alpha (S \circ T) \subseteq
 \alpha S \circ  \alpha T \circ  \alpha T \circ \alpha S
\circ  \alpha S \circ  \alpha T \dots,
   \end{equation*}    
rather than identities of the form \eqref{jr}.
Compare also identities \eqref{anrrr} - \eqref{anrrrrcon} below.
More generally, there is  the possibility of considering identities like
\begin{equation}\labbel{jrc}
 \alpha (R_1 \circ R_2) \subseteq
 \alpha R_{f(0)} \circ  \alpha R_{f(1)}\circ  \alpha R_{f(2)} \circ 
 \dots,
   \end{equation}    
where $f $ is a function with codomain $\{ 1,2\}$.   
Some arguments presented in \cite[Section 4]{uar} 
suggest that considering identities like  
\eqref{jrc} is a very natural choice. 
 \end{remark}   

\subsection*{Two  Gumm terms} \labbel{tnt} 
Let us say that a variety $\mathcal {V}$
has \emph{$2$ Gumm terms}
if $\mathcal {V}$ has terms 
$j_1$ and  $j_2$
satisfying the identities 
 (J$_L$), (J$_C$), (J$_R$) and
(J$_1$) from \eqref{ggg} in Section \ref{tolimp}; thus we are leaving out identity
(J$_2$), as we did in Lemma \ref{lem3}.
Sometimes  a variety as above is said to have
$3$ Gumm terms, since 
a trivial projection $j_0$  
 onto the first coordinate
is counted.
Notice that some authors give the definition 
of Gumm terms with the ordering of variables and of terms
 reversed. This applies to some papers of ours, too, but here
it is more convenient to maintain the analogy 
with the common condition for $3$-distributivity. 
A variety $\mathcal {V}$ has 
$2$  Gumm terms if and only if 
$\mathcal {V}$ satisfies the congruence identity
$\alpha ( \beta \circ \gamma ) \subseteq \alpha  \beta  \circ
\alpha ( \gamma \circ \beta ) $.
Notice that if we take converses and exchange $\beta$ and
$\gamma$, we get 
$\alpha ( \beta \circ \gamma ) \subseteq \alpha ( \gamma \circ \beta ) \circ
\alpha \gamma $.
See, e.~g., \cite{G,adjt,jds,ntcm} for further details
and information. 

If $S$ is a reflexive binary relation on some 
algebra, we denote by $ \overline{S} $
the smallest admissible relation containing $S$.
In the following theorem 
we shall provide bounds for expressions like, say,
$  R  (S \circ T ) \circ R W $,
rather than simply
$  R  (S \circ T ) $.
If   $0$ denotes 
 the minimal congruence, i.e., the identity relation,
then one can always obtain the latter expression from the former,
by taking $W=0$. 
The reader might always assume to be in the simpler situation
when $W=0$ with just an exception.
We need the more involved formula \eqref{3ga} below
  with the actual presence 
of $RW$  in order to get the relation
$R  (S \circ T) \circ R  (S \circ T)
\subseteq 
R  S \circ RT \circ R  (S \circ T)$
in the course of the proof of Corollary \ref{cor3} below. 
Theorem \ref{3dist}, too, allows a version with an
$RW$ factor added; we have not included this more general version
in the statement
for simplicity and since we do not need it here. 
 In detail, arguing as in the proof of identities 
\eqref{3gb} and \eqref{3gc} below, 
we get that a $3$-distributive variety  satisfies
$   R  (S \circ T ) \circ R W 
 \subseteq 
 R  S \circ  R  T \circ  R T \circ  R  (\overline{S \cup W}) $,
as well as 
$ 
   R  (S \circ T ) \circ R W 
 \subseteq 
 R  S \circ  R  T \circ  R S \circ R  (\overline{T \cup W})
$. 

\begin{theorem} \labbel{3gumm} 
If $\mathcal {V}$ has $2$  Gumm terms, then
$\mathcal {V}$ satisfies 
\begin{align}    
\labbel{3ga}  
  R  (S \circ T ) \circ R W 
& \subseteq 
 R  S \circ  R  T \circ  R  (\overline{S \cup T \cup W}),
\\
\labbel{3gb} 
   R  (S \circ RT ) \circ R W 
& \subseteq 
 R  S \circ  R  T \circ  R T \circ  R  (\overline{S \cup W}),
\\
\labbel{3gc} 
   R  (RS \circ T ) \circ R W 
& \subseteq 
 R  S \circ  R  T \circ  R S \circ R  (\overline{T \cup W}).
 \end{align}
\end{theorem}

\begin{proof}
Suppose that $(a,c^{\bullet}) \in   R  (S \circ T )  \circ R W $ with 
$ a \mathrel { R} c $ and 
$a \mathrel {S } b \mathrel {T } c \mathrel { RW} c^{\bullet} $.
 By Lemma \ref{lem3}
we have $a \mathrel {  R  S } e \mathrel {   R  T} f  $, hence  
 in order to prove \eqref{3ga}
it is enough to compute
\begin{gather*}
f 
{ =  ^{\text{(J$_R$)}}}
 j_2(j_2( \underline{a}, c,  c), j_2( b, \underline{b},  c) , \underline{c})
  \mathrel { \overline{S \cup  T \cup W} }
  j_2(j_2( \underline{b}, c,  c), j_2( b, \underline{c},  c) , \underline{c}^{\bullet})
{=  ^{\text{(J$_R$)}}} c^{\bullet}, 
\\
f=  j_2(j_2( \underline{a}, c,  c),   c , \underline{c})
 \mathrel { R } \:
  j_2(j_2( \underline{c}, c,  c),   c , \underline{c}^{\bullet})
 =  ^{\text{(J$_R$)}} c^{\bullet}.
  \end{gather*}        

To prove \eqref{3gb},
take $T=RT$ in Lemma \ref{lem3},
thus
$f \mathrel { RT } g $.
We can also assume   
$b \mathrel R c$, hence 
\begin{gather*}
 g = j_2 ( j_2 ( \underline{a},c,c), j_2 ( \underline{b}, c ,c),  \underline{c}  ) 
 \mathrel { R } 
 j_2 ( j_2 ( \underline{c},c,c), j_2 ( \underline{c}, c ,c),  c^{\bullet}  ) 
=^{\text{(J$_R$)}} c^{\bullet},
\\ 
g = j_2 ( j_2 ( \underline{a},c,c), j_2 (b, c ,c),  \underline{ c } ) 
 \mathrel { \overline{S \cup W }} 
 j_2 ( j_2 ( \underline{b},c,c), j_2 (b, c ,c), \underline{  c }^{\bullet}  ) 
=^{\text{(J$_R$)}} c^{\bullet}.
\end{gather*}    

The proof of \eqref{3gc} is similar, using the element
$g^{\bullet}$.
This time we assume $S=RS$ and $a \mathrel { R} b $.
The relations still needing a proof are
\begin{gather*}
g^{\bullet} = 
j_2 ( j_2 ( \underline{a},c,c), j_2 ( \underline{a}, b ,c),  \underline{c}  ) 
 \mathrel { R } 
 j_2 ( j_2 ( \underline{c},c,c), j_2 ( \underline{ b} , b ,c),  \underline{  c }^{\bullet}  ) 
=^{\text{(J$_R$)}} c^{\bullet},
\\
 g^{\bullet} = j_2 ( j_2 ( a,c,c), j_2 (a, \underline{ b} ,c),  c  ) 
 \mathrel { \overline{T \cup W}} 
 j_2 ( j_2 ( a,c,c), j_2 (a, \underline{ c} ,c),  \underline{  c }^{\bullet}  ) 
=^{\text{(J$_R$)}} c^{\bullet}.
\qedhere \end{gather*}    
\end{proof}

\begin{remark} \labbel{malc} 
Using the methods in the proof of 
Lemma \ref{lemid}
we could have given Maltsev conditions for the identities
\eqref{3b} - \eqref{3ga} 
by means of the existence of  terms 
satisfying appropriate equations.
Then, following the lines of the proofs
of Theorems \ref{3dist} and \ref{3gumm},  
 one can construct explicitly these terms in function
of $j_1$ and  $j_2$. 
In the case at hand, dealing directly with 
relation identities seems simpler.
However, there is an aspect dealing with terms 
which deserves  mention.

The proofs of Theorems \ref{3dist} and \ref{3gumm}
implicitly use the terms
$d_1(x,y,z) = j_1(j_1(x,y,z),y,z)$
and
$d_2(x,y,z)= j_2(j_2(x,z,z),j_2(x,y,z),z)$.
If $j_1$ and  $j_2$
are J{\'o}nsson terms, then
the terms $d_1$ and  $d_2$ 
 constitute a set of 
two \emph{directed J{\'o}nsson terms} \cite{adjt,Z}, namely,
they satisfy the equations 
\begin{equation}
\begin{aligned} \labbel{2d} 
x&=d_1(x,x,z), \quad    &  d_1(x,z,z) &= d_2(x,x,z),    & \quad  d_2(x,z,z) &= z ,
\\
x&=d_1(x,y,x) ,           & &  &         d_2(x,y,x) &= x .
 \end{aligned}    
\end{equation}    
Compare the middle equation in \eqref{2d} with (J$_C$).
Moreover, in comparison with
Remark \ref{asym}, 
 notice that here there is full symmetry between $d_1$ and  $d_2$.
 
Among other, we have showed that a $3$-distributive 
variety has two  directed J{\'o}nsson terms.
This is the poor man's  version of the far more general result  
proved by Kazda,  Kozik,  McKenzie and Moore \cite{adjt} that
\emph{every} congruence distributive variety 
has a (generally longer) set of directed J{\'o}nsson  terms. 

Yet it is interesting to check which
parts of Theorem \ref{3dist} 
 follow just from the existence of 
two  directed J{\'o}nsson terms.
The existence of two  directed J{\'o}nsson terms
implies identity \eqref{3a}, limited to  the case 
when $R$ is assumed to be a tolerance,
by the case  $\ell =2$
in equation (3.1) in \cite[Proposition 3.1]{jds}.
 Notice that the number of terms is counted in
a different way in \cite{jds}, including also two 
trivial projections; here we adopt the counting
convention from \cite{adjt}.

 On the other hand, there are varieties
having two directed J{\'o}nsson terms
and which are not $3$-distributive:
see \cite[p.\ 11]{jds} and \cite{baker}.
Since
the identities \eqref{3b} and \eqref{3c}
do imply $3$-distributivity, we get that 
the existence of 
two  directed J{\'o}nsson terms
does not imply \eqref{3b} and \eqref{3c}.
In particular, our arguments go far beyond 
 the simple proof of the existence of directed J{\'o}nsson
terms and exploit
special particularities of $3$-distributive
varieties.

Similar remarks apply to 
Theorem  \ref{3gumm},
with reference to \emph{directed Gumm terms} \cite{adjt}.
\end{remark}

If $R$ is  a binary
relation, we let $R^*$ denote the transitive closure of $R$.

\begin{corollary} \labbel{cor3} 
 (1) A $3$-distributive variety satisfies
\begin{align} \labbel{3aa}      
R^*(S \circ T) ^* & = (RS \circ RT) ^* \text{ and, more generally,}
\\ 
\labbel{3bb}
 (R \circ V)^*(S \circ T) ^* & = (RS \circ RT \circ VS \circ VT) ^* .
 \end{align}     

(2)
If $\mathcal {V}$ has $2$  Gumm terms,
 then $\mathcal {V}$   satisfies
\begin{align}  \labbel{3cc}   
R^*(S \circ T) ^*   & = (RS \circ RT) ^* \circ 
 R  (\overline{S \cup T}),  
\\ 
 \labbel{3ccb}   
R^*(RS \circ T) ^*   & = (RS \circ RT) ^* , \text{ and } 
\\ 
\labbel{3dd}
 R ^*S  ^* & = (RS) ^* .
 \end{align}      
\end{corollary} 

\begin{proof}
We first prove identity \eqref{3dd} in (2).
Of course, 
\eqref{3dd} is the particular case
$S= T$ of \eqref{3cc},
however, it seems convenient to obtain  
the latter identity as a consequence of the former, hence
let us  prove \eqref{3dd}.
Let $S^n$
denote the iterated relational composition of $S$ 
with itself  with a total of $n$ factors and recall that
 $0$ denotes 
 the minimal congruence.
By taking $W=0$ and  
$S^n$ 
 in place of both $S$ and $T$ 
in \eqref{3ga},  
we get
 $RS^{2n} = R(S^{n} \circ S^{n})
\subseteq 
RS^{n} \circ RS^{n} \circ RS^{n}$. 
Since 
$RS^* = \bigcup _{m \in \omega } RS^{m} $,
we get by induction that
$RS^* \subseteq (RS)^*$. 
By taking $R^*$ in place of $R$
in the above identity and then exchanging the role of 
$R$ and $S$ we get  
$R^*S^* \subseteq (R^*S)^* \subseteq 
((RS)^*)^* = (RS)^* $.
The reverse inclusion in  \eqref{3dd}  is trivial.

Having proved \eqref{3dd}, then
by taking $S \circ T $ in place of $S$ in 
\eqref{3dd}, we get
$R^*(S \circ T) ^*  
= 
( R  (S \circ T))^*$.
By taking $W=S \circ T$
in  \eqref{3ga}, we get
$R  (S \circ T) \circ R  (S \circ T)
\subseteq 
R  S \circ RT \circ R  (S \circ T)$.
Then by induction, always 
factoring out the first compound factor on the left, 
we get
\begin{equation}\labbel{eqeq}
  ( R  (S \circ T))^* 
\subseteq 
(R  S \circ RT)^* \circ R  (S \circ T).
  \end{equation}    
Applying again 
 \eqref{3ga} with 
$W=0$, we get 
 $R  (S \circ T) \subseteq R  S \circ RT \circ R ( \overline{S \cup T} ) $. 
Summing everything up, 
we get the $ \subseteq $ inclusion in
 \eqref{3cc}. The reverse inclusion is trivial,
since $S \circ T \supseteq \overline{S \cup T} $. 

Applying \eqref{3gc} with $W=0$,
we get 
$    R  (RS \circ T ) 
 \subseteq 
 R  S \circ  R  T \circ  R S \circ R T$.
By \eqref{eqeq} with $RS$
in place of $S$    we get the
$ \subseteq $ inclusion in
 \eqref{3ccb}. Again, the reverse inclusion is trivial,
We have proved (2).

Now 
\eqref{3aa} is immediate from
\eqref{3cc} 
and \eqref{3a}.
Identity \eqref{3bb}
follows by repeated applications of 
 \eqref{3aa}.    
Of course, identity 
\eqref{3aa}  can be also proved directly
from \eqref{3a} by using only a minimal part of the
above arguments.
 \end{proof}

\subsection*{Further comments} \labbel{fur3} 
Theorems  \ref{3dist} and \ref{3gumm}  can be improved in many ways.
For example,  notice that the ``middle'' element
$f= j_1(j_1(a,c,c),c,c) $ in the proof of Theorem \ref{3dist}
  does not depend on 
$b$. Hence the proof provides a bound for
$R (S_1 \circ T_1 )(S_2 \circ T_2)(S_3 \circ T_3) \dots $.
Moreover, the element
$f$ is the same in the proofs of
\ref{3dist} and \ref{3gumm}.
In addition, the arguments in the proof of 
Theorem \ref{3gumm} clearly apply to
$3$-distributive varieties, as well
(here take $c^{\bullet}=c$ or, which is the same, $W=0$).
Moreover, in the notations from the proof of 
Theorem \ref{3dist},
we have
$a= ^{\text{(J$_L$)}} j_1(j_1(a, \underline{a},c), \underline{a},c)
 \mathrel { R}
 j_1(j_1(a, \underline{c},c), \underline{c},c) = f $.  
Summing everything up, we get 
items (1) and (2) in the following proposition.

As for (3) below, the definition of two directed J{\'o}nsson terms
recalled in \eqref{2d} from  Remark \ref{malc}
can be extended to any chain of terms; see 
\cite{adjt} for full details.  Then, merging the above arguments with the 
proof of \cite[Proposition 3.1]{jds}, we get a proof of (3) below.

\begin{proposition}  \labbel{ammennicol}
  \begin{enumerate}   
 \item  
A $3$-distributive variety
satisfies 
\begin{multline*}
R  (S_1 \circ T_1 )(S_2 \circ T_2)(S_3 \circ T_3) \dotsc
\subseteq 
 R (R S_1 \circ R T_1)(R S_2 \circ R T_2)
(R S_3 \circ R T_3) \dotsc  \circ 
\\
R (R S_1 \circ R T_1)
(R T_1 \circ R S_1)
(R S_2 \circ R T_2)(R T_2 \circ R S_2) \dots
 (\overline{S_1 \cup T_1})
(\overline{S_2 \cup T_2})\dotsc 
 \end{multline*}   

\item
A variety with $2$ Gumm terms satisfies
\begin{multline*}
R  (S_1 \circ T_1 )(S_2 \circ T_2)(S_3 \circ T_3) \dotsc
\subseteq 
 R (R S_1 \circ R T_1)(R S_2 \circ R T_2)
(R S_3 \circ R T_3) \dotsc  \circ 
\\
R 
 (\overline{S_1 \cup T_1})
(\overline{S_2 \cup T_2}) 
(\overline{S_3\cup T_3}) 
\dotsc 
 \end{multline*}   
\item
A variety with $n$ ($n+2$ in the terminology from
\cite{jds}) directed J{\'o}nsson terms 
satisfies
\begin{equation*}
\Theta  (S_1 \circ T_1 )(S_2 \circ T_2)(S_3 \circ T_3) \dotsc
\subseteq 
( \Theta (\Theta S_1 \circ \Theta T_1)(\Theta S_2 \circ \Theta T_2)
(\Theta S_3 \circ \Theta T_3) \dots )^n  .
  \end{equation*}    
 \end{enumerate} 
 \end{proposition}

By Corollary \ref{cor3} 
we get an affirmative solution to Problem 2.12 in Gyenizse and  Mar\'oti \cite{gm}
in the special cases of $3$-distributive varieties
and of varieties with $2$  Gumm terms.
Recall that a \emph{preorder}  is a
 reflexive and transitive relations. 

\begin{corollary} \labbel{corgm} 
If $\mathbf A$ belongs to a
$3$-distributive variety (a variety with $2$ Gumm terms)
then the lattice of  admissible preorders of $\mathbf A$ 
is congruence distributive 
(respectively, congruence modular).
\end{corollary}

\begin{proof} 
Immediate from identities \eqref{3aa} and \eqref{3ccb}.   
\end{proof}

\section{Assuming $n$-permutabiliy} \labbel{np}

The length of a chain of iterated relational compositions
is bounded in a congruence  $n$-permutable variety.
Hence, assuming also congruence 
distributivity,  we get bounds for the value of $k$
in the identity \eqref{jr} from the introduction.
This observation shall be dealt with in 
Theorem  \ref{npermd} below. Since we are dealing with
admissible relations, not congruences, the bound is not exactly 
$n$, in general.  See Corollary \ref{equiv} below. 
  Recall that a variety $\mathcal {V}$ is \emph{congruence $n$-permutable} if
the congruence identity
$ \beta \circ \gamma \circ \beta \dotsc = 
\gamma \circ \beta \circ \gamma \dotsc $
($n$ factors on each side) holds in $\mathcal {V}$.

We shall use  the following conventions, in order to 
make the notation more compact.

\begin{convention} \labbel{conv2}The expression
$ S \circ T \circ  \stackrel{m}{\dots} \ $ denotes
$ S \circ T \circ S \dots$
with $m$ factors, that is, 
with $m-1$ occurrences of $\circ$.
If, say, $m$ is odd, we sometimes  write
$ S \circ T \circ  {\stackrel{m}{\dots}} \circ S $
when we want to make clear that $S$ is the last factor
in the expression.
We shall use the above notation even when
$R$ or $T$ are replaced with compound factors,
or even when they do not strictly alternate, such as in the expression
\begin{equation}\labbel{prv}
     \alpha S \circ \alpha T \circ  \alpha T \circ \alpha S 
\circ \alpha S \circ \alpha T \circ {\stackrel{m}{\dots}} 
   \end{equation}
In any case, the number above the dots always indicates 
the number of $\circ$'s minus one, that is, the number
of ``factors'', with the provision that compound
factors are always counted as one factor.  For example, if $m=10$
in \eqref{prv}, the expression reads  
\begin{equation*}
     \alpha S \circ \alpha T \circ  \alpha T \circ \alpha S 
\circ \alpha S \circ \alpha T \circ  \alpha T \circ \alpha S 
\circ \alpha S \circ \alpha T   .
  \end{equation*}
In any case, when we write only two factors on the right as in 
$ \alpha S \circ \alpha T \circ  \stackrel{m}{\dots} \ $ or in
$ \alpha S \circ  \alpha T \circ  {\stackrel{2m-1}{\dots}} \circ \alpha S $,
we always mean that these two factors alternate
as $ \alpha S \circ  \alpha T  \circ \alpha S \circ  \alpha T  \dots \ $

Moreover, recall that if $R$ is a reflexive binary relation on some 
algebra, $ \overline{R} $
denotes the smallest admissible relation containing $R$
and that $R^*$ denotes the transitive closure of $R$.
Recall also the definition of $R^m$. In  the above notation,  $R^m $ is 
$R \circ R \circ {\stackrel{m}{\dots}}$.
 \end{convention}

Parts of the following proposition 
are known. See,  e.~g.,
Werner \cite{W},   Hagemann and  Mitschke
\cite{HM} and  reference [3] quoted  there.
However it seems that some parts of the  proposition 
are new.

\begin{proposition} \labbel{nperm}
Suppose that $n \geq 2$ and $\mathcal {V}$ is a variety.
Each of the 
following
identities is equivalent to 
congruence $n$-permutability of $\mathcal {V}$.
\begin{align} \labbel{ne}
R_1 \circ R_2 \circ \dotsc \circ R_{n}
&\subseteq 
 \overline{R_1 \cup R_2}  \circ  \overline{R_2 \cup R_3} 
\circ \dotsc \circ
  \overline{R_{n-1} \cup R_n } \,,
\\
\labbel{nra}
R^n &\subseteq  R ^{n-1},\\
\labbel{nr}
R^* &= R ^{n-1},
\\ 
\labbel{nrr}
(S \circ T)^* &= S \circ T\circ S \circ T \circ  {\stackrel{2n-2}{\dots}} \circ S \circ T ,
\\ 
\labbel{nrrr}
(S \circ T)^* &= S \circ T\circ T \circ S \circ  S \circ T
 \circ  {\stackrel{2n-2}{\dots}},
\\ 
   \labbel{nrrrr}
 S \circ T\circ T \circ S \circ  S  \circ T \circ  {\stackrel{2n-2}{\dots}}
&=
 T \circ S \circ  S \circ T \circ  T \circ S \circ  {\stackrel{2n-2}{\dots}},
\\
 \labbel{nrup}
S \circ T \circ  {\stackrel{2n-1}{\dots}} \circ S
& \subseteq 
T \circ S \circ  {\stackrel{2n-2}{\dots}} \circ S ,
\\
 \labbel{nrupp}
S \circ T \circ  {\stackrel{2n-1}{\dots}} \circ S
&=
T \circ S \circ  {\stackrel{2n-1}{\dots}} \circ T .
\end{align}   

In all the above identities we can equivalently
let $R_1,\dots, R, S, T$ be either reflexive and admissible relations 
or tolerances.

Moreover, identities \eqref{ne}, \eqref{nrrr} and  \eqref{nrrrr} are still equivalent to
congruence $n$-permutability if we let  $R_1,\dots,  S, T$ be congruences.
\end{proposition}

\begin{proof}
We shall give a proof which works  in both cases, either when
$R_1,\dots, R, \allowbreak  S, T$  are reflexive and admissible relations
or when they are tolerances. 
Notice that in  the identities  \eqref{nr} -
 \eqref{nrrr} 
the left-hand side is always larger than the right-hand side, 
hence it is enough to prove  the reverse  inclusion.
Let ($n$-perm) denote congruence $n$-permutability.  We shall prove
that ($n$-perm) $\Rightarrow $  \eqref{ne} for varieties and that
\eqref{ne} $\Rightarrow $  
\eqref{nra} $\Rightarrow $   \eqref{nr}  
 $\Rightarrow $  \eqref{nrrr}   $\Rightarrow $  \eqref{nrrrr}
$\Rightarrow $ ($n$-perm) and 
    \eqref{nr} 
$\Rightarrow $  \eqref{nrr} $\Rightarrow $
\eqref{nrup} $\Rightarrow $  \eqref{nrupp} $\Rightarrow $ 
\eqref{nra} 
hold in every algebra.

If, say, $\mathcal {V}$ is a congruence $3$-permutable variety, then,
by \cite{HM}, there are terms
$t_1$ and  $t_2$ such that $x=t_1(x,y,y)$,
$t_1(x,x,y) = t_2(x,y,y)$  and
 $t_2(x,x,y) =y$ are equations valid in $\mathcal {V}$.
If $a \mathrel { R_1} b \mathrel { R_2} c \mathrel { R_3} d$,
then
$a= t_1( \underline{a},b, \underline{b})
 \mathrel { \overline{R_1 \cup R_2} }
 t_1( \underline{b},b, \underline{c}) =
 t_2( \underline{b},c, \underline{c}) 
\mathrel { \overline{R_2 \cup R_3} } 
t_2( \underline{ c},c, \underline{d}) =d $,
hence \eqref{ne} follows in the case $n=3$. In general,
for $n \geq 2$, the same argument shows that every 
congruence $n$-permutable variety satisfies \eqref{ne}.

Taking $R_1= R_2 = \dots = R$ in \eqref{ne}, we get 
\eqref{nra}. Then
\eqref{nr} follows from a trivial induction.

 In order to prove that \eqref{nr} implies  \eqref{nrrr},
observe that  $S \circ T$ is a reflexive and admissible relation
which contains both $S$ and $T$, thus
$ S \circ T \supseteq  \overline{S \cup T}$.
Since  $S, T \subseteq \overline{S \cup T} $,
hence 
$S \circ  T \subseteq \overline{S \cup T} \circ \overline{S \cup T}$,
we get 
$(S \circ T)^* = (\overline{S \cup T})^*$.
By taking the  admissible relation   
$\overline{S \cup T}$ in place of $R$ in \eqref{nr} and using alternatively
 $ \overline{S \cup T} \subseteq S \circ T $ and
$ \overline{S \cup T} \subseteq T \circ S $,
we get \eqref{nrrr}. Notice that if $S$ and $T$ are tolerances,
then $R=\overline{S \cup T}$ is a tolerance, hence the proof works
also in the case when we are dealing with tolerances.
On the other hand, notice that   $S \circ T$
is not necessarily a tolerance, even when $S$ and $T$ are.

The implication  \eqref{nrrr} $\Rightarrow $     \eqref{nrrrr}
is obvious, since $(S \circ T)^* = (T \circ S)^*$. 

By taking 
$S$ and $T$ congruences in 
\eqref{nrrrr}, we get 
$ \alpha \circ \beta \circ {\stackrel{n}{\dots}}  =
 \beta \circ \alpha \circ {\stackrel{n}{\dots}} $,  that is,
congruence $n$-permutability. 
Notice that, since congruences are transitive,
$n-2$ factors annihilate on each side of \eqref{nrrrr},
hence we end up with $n$ actual factors on each side.   
Notice that we have only used transitivity of tolerances, we do not need
symmetry.

The proof  that \eqref{nr}
implies \eqref{nrr} is similar to the proof that
\eqref{nr}
implies \eqref{nrrr}.
Take $R=\overline{S \cup T}$ and always use
 $ \overline{S \cup T} \subseteq S \circ T $.
Notice that, were 
$S$ and $T$ reflexive and admissible relations,
we could have proceeded in a simpler way, since
$S \circ T$ is reflexive and admissible, too, hence
\eqref{nrr} is the special case of \eqref{nr}
obtained by considering $S \circ T$   
in place of $R$. However, the previous argument works also for tolerances.

The implication
\eqref{nrr} $\Rightarrow $  \eqref{nrup}
is trivial, exchanging the role of $S$  and $T$, since
$S \circ T \circ  {\stackrel{2n-1}{\dots}} \circ S
\subseteq (S \circ T)^*$. 

By identity  \eqref{nrup} we get
$S \circ T \circ  {\stackrel{2n-1}{\dots}} \circ S
\subseteq 
T \circ S \circ  {\stackrel{2n-2}{\dots}} 
\circ S
\subseteq 
T \circ S \circ  {\stackrel{2n-1}{\dots}} \circ S
\circ T$,
 thus by symmetry we get
equality of the outer expressions,  that is,  \eqref{nrupp}.  

Recall that 
$0$ denotes 
 the minimal congruence.
Taking $S=R$ and $T=0$   
in \eqref{nrupp} we get back \eqref{nra}.

Now we prove the last statement.
If some identity holds for reflexive and admissible relations, 
then it trivially holds  for congruences,
hence congruence $n$-permutable varieties satisfy  
the versions of \eqref{ne}, \eqref{nrrr} and \eqref{nrrrr}
when the variables range on such a restricted scope.
We already mentioned that \eqref{nrrrr} relative to congruences
does imply congruence $n$-permutability
and the argument is the same for 
\eqref{nrrr}, namely, take $S$ and $T$ congruences.
Finally, if $R_1=R_3= \dots = \alpha $
and $R_2=R_4= \dots = \beta  $ are congruences, then
\eqref{ne} implies $\alpha  \circ \beta \circ {\stackrel{n}{\dots}}   
\subseteq (\overline{ \alpha \cup \beta }) ^{n-1}  \subseteq 
\beta \circ \alpha \circ \alpha \circ \beta  \circ \beta \circ \alpha 
\circ {\stackrel{2n-2}{\dots}}  
=
\beta  \circ  \alpha  \circ {\stackrel{n}{\dots}}  $.
\end{proof}

\begin{remarks} \labbel{ror}
(a) The proof of Proposition \ref{nperm}
shows that, for every $n$, the identities 
\eqref{nra} -  \eqref{nrrr}, \eqref{nrup} - \eqref{nrupp}
are equivalent for every algebra.
Indeed,  identity \eqref{nrrr} implies \eqref{nr}, by taking
  $T=0$, all the rest has been mentioned in the proof of \ref{nperm}. 

(b) We now justify the assertion in the introduction
that the equivalence between 
congruence permutability and 
the identity $R \circ R =R$ is non trivial, 
in the sense that it holds for varieties, but not for single algebras.

On one hand,  in every algebra,
the identity $R \circ R =R$ for relations, or even 
the identity $\Theta \circ \Theta = \Theta$ 
for tolerances,
 do imply  
congruence permutability.
This follows from the implication
\eqref{nra} $\Rightarrow $  \eqref{nrrr}
in Proposition \ref{nperm}   
in the case $n=2$. 
More directly. if $\alpha$ and $\beta$ are congruences,
take $\Theta= \overline{ \alpha \cup \beta } $,
the smallest admissible relation containing $\alpha$ and $\beta$, 
getting 
$\alpha \circ \beta  \subseteq \Theta \circ \Theta = \Theta 
= \overline{ \alpha \cup \beta } \subseteq \beta \circ \alpha $.
Compare the proof of \cite[Proposition 13(a)]{contol}.  
 
On the other hand, let $A= \{ 0,1,2\} $
and $\mathbf A$ be the algebra on $A$ 
with only one unary operation $g$
defined by $g(0)=1$, $g(1)=g(2)=2$.
The lattice 
$ \con (\mathbf A)$    
 has $3$ elements, hence consists of pairwise permutable
congruences, 
the relation
$R = \{ (0,0), (1,1), (2,2),(0,1), (1,2)\} $
is reflexive and admissible,
but 
$R \circ R =R$ fails.
\end{remarks}

\begin{corollary} \labbel{equiv}
For every variety $\mathcal {V}$ and
every $n \geq 2$, the following conditions are equivalent.
 \begin{enumerate} 
  \item 
$\mathcal {V}$ is congruence $n$-permutable.
\item
$\mathcal {V}$ is admissible-preorder $n$-permutable.
\item
$\mathcal {V}$ is admissible-relation $2n-1$-permutable.
\item
$\mathcal {V}$ is tolerance $2n-1$-permutable. 
  \end{enumerate}  
 \end{corollary} 

\begin{proof}
The equivalence of (1) and (2) follows from identity
\eqref{nrrrr}.

 The equivalence of (1), (3) and (4) follows from identity
\eqref{nrupp}.
 \end{proof}  

In particular, apart from the degenerate case,
congruence $n$-permutability 
is equivalent to admissible-relation $n$-permutability
if and only if $n=2$.   
Notice that the above results do not exactly clarify
the meaning of admissible-relation $n$-permutability 
for $n$ even. 
In fact, by identity \eqref{nrr} 
in Proposition \ref{nperm},
we get that 
every congruence $n$-permutable variety satisfies 
$S \circ T\circ S \circ T \circ  {\stackrel{2n-2}{\dots}}
= T\circ S \circ T \circ S \circ  {\stackrel{2n-2}{\dots}}$;
however, we do not know the exact consequences of this last identity.

Recall the notational conventions established in the abstract and in 
\ref{conv2}.

\begin{theorem} \labbel{npermd}
Suppose that $n \geq 2$ and $\mathcal {V}$ is a variety.
Each of the 
following
identities is equivalent to the conjunction of
congruence distributivity and  $n$-permutability of $\mathcal {V}$.
\begin{align} 
\labbel{anrr}
 (\Theta  (S \circ T)^*)^* 
&=
 \Theta  S \circ \Theta  T\circ \Theta  S \circ \Theta  T
 \circ {\stackrel{2n-2}{\dots}} \circ  \Theta  S \circ \Theta  T,
\\ 
\labbel{anrrr}
(\Theta  (S \circ T)^*)^*
 &= 
\Theta  S \circ \Theta  T\circ \Theta  T \circ \Theta  S 
\circ \Theta   S \circ \Theta  T  \circ {\stackrel{2n-2}{\dots}},
\\
   \labbel{anrrrrtol}
\Theta ( S \circ T\circ T \circ S \circ  S  \circ  {\stackrel{2n-2}{\dots}})
& \subseteq 
 \Theta T \circ \Theta S \circ \Theta S \circ \Theta T \circ
\Theta  T \circ \Theta S \circ  {\stackrel{2n-2}{\dots}},
\\   
\labbel{anrrrrcon}
\alpha ( S \circ T\circ T \circ S \circ  S   \circ  {\stackrel{2n-2}{\dots}})
& =
 \alpha T \circ \alpha S \circ \alpha S \circ \alpha T \circ
\alpha  T \circ \alpha S \circ  {\stackrel{2n-2}{\dots}},
\\
 \labbel{anrupp}
 \alpha (S \circ   T \circ  {\stackrel{2n-1}{\dots}} \circ  S)
&=
\alpha T \circ \alpha S \circ  {\stackrel{2n-1}{\dots}} \circ \alpha  T .
\end{align}   

The variables $S$ and $T$ 
can be equivalently taken to be tolerances in
\eqref{anrr} - \eqref{anrupp} 
and
congruences in \eqref{anrrr} - \eqref{anrrrrcon}.
\end{theorem}

\begin{proof}
The 
$  \supseteq $ inclusions  in   
\eqref{anrr},  \eqref{anrrr} and
\eqref{anrrrrcon}
are obvious.
Notice that in \eqref{anrrrrcon}  we need transitivity of $\alpha$.

 If $\mathcal {V}$ is congruence distributive,
then, by  Kazda,  Kozik,  McKenzie and Moore \cite{adjt}  and by
 \cite[Proposition 3.1]{jds},
$\Theta  (S \circ T )^* \subseteq 
(\Theta  S \circ \Theta   T )^* $.  
If in addition $\mathcal {V}$ is congruence $n$-permutable,  we can apply 
the identities \eqref{nrr}, \eqref{nrrr} and   \eqref{nrupp}  with $\Theta  S$
and $\Theta  T $ in place of, respectively,
$S$ and $T$, obtaining a bounded number of factors on the right.
Hence \eqref{anrr} - \eqref{anrupp} follow from the assumptions.
To prove the $ \supseteq $ inclusion in \eqref{anrupp} use 
\eqref{nrupp} in order to get 
 $\alpha T \circ \alpha S \circ  {\stackrel{2n-1}{\dots}} \circ \alpha  T 
= 
\alpha S \circ \alpha T \circ  {\stackrel{2n-1}{\dots}} \circ \alpha  S
\subseteq
 \alpha (S \circ   T \circ  {\stackrel{2n-1}{\dots}} \circ  S)$,
again since $\alpha$ is transitive.

Conversely, by taking $\Theta$, $S$ and $T$ congruences in each of
\eqref{anrr} - \eqref{anrupp}, we get identities implying
congruence distributivity by \cite{J}.
If we take $\Theta = \alpha = 1$, that is, the largest congruence,
in  \eqref{anrr} - \eqref{anrupp}, we get one of
\eqref{nrr} - \eqref{nrrrr}
or \eqref{nrupp}.
These identities imply congruence $n$-permutability by
Theorem \ref{nperm}.
\end{proof}

\begin{theorem} \labbel{npermm}
Suppose that $n \geq 2$ and $\mathcal {V}$ is a variety.
Each of the 
following
identities is equivalent to  the conjunction of
congruence modularity and $n$-permutability of $\mathcal {V}$.
\begin{align}
\labbel{nram}
\Theta R^n &\subseteq  (\Theta R) ^{n-1},
\\
\labbel{nrm}
\alpha R^* &= ( \alpha R )^{n-1},
\\ 
\labbel{nrrm}
\alpha (S \circ  \alpha T)^* & =
\alpha S \circ \alpha T\circ \alpha S \circ \alpha T \circ  {\stackrel{2n-2}{\dots}} 
\circ \alpha S \circ \alpha T ,
\\ 
\labbel{nrrrm}
 \alpha (S \circ \alpha T)^* &= \alpha  S \circ 
\alpha T\circ  \alpha T \circ \alpha  S \circ \alpha  S \circ \alpha T
 \circ  {\stackrel{2n-2}{\dots}},
\\
 \labbel{nruppm}
\alpha (S \circ \alpha T \circ  {\stackrel{2n-1}{\dots}}
\circ \alpha T \circ S)
&=
\alpha T \circ \alpha S \circ  {\stackrel{2n-1}{\dots}} \circ \alpha  T .
\end{align}    
The variables $R$, $S$ and $T$ 
can be equivalently taken to be tolerances 
in \eqref{nram} - \eqref{nruppm}
and
congruences in \eqref{nrrrm}.
\end{theorem}  

\begin{proof}
We have from
\cite[Theorem 1]{ricmc}
that a congruence modular variety satisfies 
$\Theta R^* \subseteq (\Theta R)^*$.
Again, this result relies heavily on  \cite{adjt}.
By congruence $n$-permutability and  \eqref{nr}
we get the nontrivial inclusions in 
\eqref{nram} and  \eqref{nrm}.

By taking $R=S \circ \alpha T$
in  \eqref{nrm} we get
$\alpha (S \circ  \alpha T)^* = 
(\alpha (S \circ  \alpha T))^{n-1} $.
But  
$\alpha (S \circ  \alpha T) = \alpha S \circ  \alpha T$,
hence we get
 $\alpha (S \circ  \alpha T)^* = 
(\alpha S \circ  \alpha T)^* $.
Then, as usual by now,
by applying Theorem \ref{nperm}
with $\alpha S$ in place of $S$ and
 $\alpha T$ in place of $T$,
we get the nontrivial inclusions in 
\eqref{nrrm} -  \eqref{nruppm}.

Taking 
$R= \beta \circ \alpha \gamma $
in \eqref{nrm} 
and using again 
$\alpha (S \circ  \alpha T) = \alpha S \circ  \alpha T$
we get congruence modularity.
Using \cite{D},
we prove in a similar way that 
\eqref{nram} implies congruence modularity.

The above arguments work when $R$, $S$ and $T$ 
are assumed to be reflexive and admissible relations.
If we assume  that $R = \Theta $, $S= \Psi$ and $T= \Phi$ 
are tolerances, argue as above using 
$R = \Theta = \alpha \Phi \circ \Psi \circ \alpha \Phi$
or 
$R = \Theta  = \alpha \gamma \circ \beta \circ \alpha \gamma $
notice that these are tolerances and that
$ \alpha (\alpha \Phi \circ \Psi \circ \alpha \Phi) =
\alpha \Phi \circ \alpha \Psi \circ \alpha \Phi$.

 By taking $S$ and $T$ congruences,
we have that \eqref{nrrm} - \eqref{nruppm}, too,
imply congruence modularity.
As in the proof of Theorem \ref{npermd},
the identities \eqref{nram} -  \eqref{nruppm}  imply congruence $n$-permutability,
 taking $\Theta= \alpha =1$
and considering the corresponding identities in Theorem \ref{nperm}.
 \end{proof}

\begin{remarks} \labbel{more}
(a) There are obviously many other intermediate equivalent
conditions in Proposition \ref{nperm}, Theorems \ref{npermd}
and \ref{npermm}. In order
to keep the statements
within a reasonable length,
  we have not explicitly stated such equivalent conditions.
For example, for every $m \geq n$,  
the identity  $R^m
\subseteq 
R ^{n-1} $
is equivalent to congruence $n$-permutability, 
as follows from the proof
of \ref{nperm}.
More generally, we can equivalently augment the number of factors
on the left in \eqref{nrrrr} - \eqref{nrupp}, \eqref{anrrrrtol} - \eqref{anrupp}, 
provided we replace $=$ with $\subseteq $.   
Similarly, we can replace transitive closure 
 with a sufficiently long iteration of compositions in 
each of \eqref{nr} - \eqref{nrrr}, \eqref{anrr}, \eqref{anrrr}.
Moreover, in \eqref{anrr} - \eqref{anrrr} 
we can replace    
$(\Theta  (S \circ T)^*)^*
= 
\Theta  S \circ \Theta  T\circ \dots $
with either
$(\Theta  (S \circ T))^*
= 
\Theta  S \circ \Theta  T\circ \dots $
or
$\Theta  (S \circ T)^*
\subseteq  
\Theta  S \circ \Theta  T\circ \dots $
Similar remarks apply to Theorem \ref{npermm}.
We leave details to the interested reader.

(b) Congruence identities characterizing the conjunction of 
$n$-permutabilty  and  distributivity or
modularity appeared in  \cite[Propositions 4 and 5]{np}. 

(c) By equations \eqref{3aa} [\eqref{3dd} and \eqref{3ccb}]
in Corollary \ref{cor3} we have that 
if the assumption of congruence distributivity [congruence modularity]
is strengthened to $3$-distributivity [to the existence of two Gumm terms],
then 
we can allow $\alpha$ and $\Theta$ to be  reflexive and admissible relations
in the identities \eqref{anrr} - \eqref{anrupp} [\eqref{nram} -  \eqref{nruppm}],
provided we replace $=$ by $ \subseteq $.
   \end{remarks}

\section{Implication algebras and a $4$-ary  near-unanimity term}
 \labbel{ianu}

We  denote by $+$, $\cdot$ and $'$
the operations of a Boolean algebra.  
The variety $\mathcal I$ of \emph{implication algebras}
is the variety generated by polynomial reducts of Boolean algebras
in which $i(x,y) = xy'$ is the only basic operation.
Equivalently, $\mathcal I$ is the variety of algebras with a binary
operation $i$ which satisfies all the equations satisfied 
by the term $xy'$ in     Boolean algebras.
A more frequent description of implication algebras
uses the term $x'+y$, instead,  but Boolean duality 
implies that
(if we reverse the order of variables)
 we get the same variety. 
 Mitschke \cite{M} showed that $\mathcal I$ 
 is $3$-distributive, not $2$-distributive,
congruence $3$-permutable and not congruence permutable.

In the above notations,
$i(x,i(y,z)) $
represents the Boolean
term 
$f(x,y,z) = x(yz')' = x(y'+z)$.  
Sometimes it is
simpler to deal with the corresponding  reduct $\mathcal I^-$
of Boolean algebras. 
Namely, $\mathcal I^-$
is the variety generated by reducts of Boolean algebras 
having  $f$  as the only basic operation.
The varieties 
$\mathcal I$ and 
 $\mathcal I^-$ have many properties in common,
for example,
$\mathcal I^-$ is still
$3$-distributive, 
congruence $3$-permutable  and, obviously, 
 not $2$-distributive and not congruence permutable.
The terms
$j_1=f(x,f(x,y,z),z) =x(x'+yz'+z)= x(y+z)$ 
and 
$j_2 = f(z,y,x)= z(y'+x)$
are J{\'o}nsson terms witnessing
$3$-distributivity of $\mathcal I^-$.
The terms $f$
and $j_2$
are Hagemann-Mitschke  
 terms \cite{HM}
for congruence $3$-permutability of $\mathcal I^-$. 
 
On the other hand, $\mathcal I^-$ is much simpler to deal with.
For example,  free algebras in 
$\mathcal I^-$ 
are much smaller 
and more easily described
than 
free algebras in 
$\mathcal I$. 
Further details about 
$\mathcal I^-$ can be found 
in the former version \cite{ia} of this note.
 Now \cite{ia} is largely subsumed by 
the present version.   

By Theorem  \ref{npermd} and the above comments, 
both
 $\mathcal I$ and $\mathcal I^-$ satisfy the identities
\eqref{anrr} and \eqref{anrrr}
with four factors on the right 
and moreover we can take an admissible relation 
$R$ in place of $\Theta$ in these identities,
by Remark \ref{more}(c). 
A direct proof 
of related identities 
 appeared in  
 \cite{ia}.
By Theorem \ref{tol2}, we get that 
 $\mathcal I$ fails to satisfy the identity \eqref{id}.
Again, a direct proof appeared in \cite{ia}.
Some features of the counterexamples presented in 
\cite{ia}
are reported  below.

 \begin{example} \labbel{anuo} 
Let 
$\mathbf 2$ be the
 $2$-elements 
 Boolean algebra and
consider the following elements of  
$2^5$:
$   x = (1,1,1,0,0)$;
 $   y = (1,0,0,1,0)$ and
$   z = (0,1,0,1,1)$.
Let $ A$ be the subset of 
$2^5$ consisting of those elements which are $\leq$
than at least  one among   
$   x $,
 $   y $ or
$   z $.
 Then $ {\mathbf A}= ( A,i)$
is an implication algebra.
Let $ \alpha$ be the kernel of the second projection, 
$ \beta$ be the 
intersection of the kernels
of the first and of the fifth projections,
 $ \gamma$ be the intersection of the kernels
of the third and of the fourth projections.

Let $  \Psi$ be the binary relation on $ A$
defined as follows:
two elements $a,b \in  A$ are $  \Psi$-related if and only if 
at least one of the following conditions holds:

(a) both $a \leq  x$ and   $b \leq  x$, or

(b) (either $a \leq   y$ or   $a \leq  z$, possibly both) and
(either $b \leq  y$ or  $b \leq   z$, possibly both).

The relation $  \Psi$ is trivially symmetric; 
$ \Psi$ is also reflexive, since, by construction, 
 every element of $  A$  
is $\leq$ than either $x$, $y$ or $z$. We claim that $ \Psi$ 
is  admissible in $ {\mathbf A}$, thus a tolerance. Indeed,
if $a \mathrel { \Psi} c $ is witnessed by (a),
then $a b' \leq a \leq  x$ and $ cd' \leq c  \leq  x$,
for all $b,d \in  A$, hence 
$a b' \mathrel { \Psi} cd'$
(we do not even need the assumption that 
$ c $ and $ d $ are $  \Psi$-related).
Similarly, 
if $a \mathrel { \Psi} c$ is given by (b),
then $a b' $ is $ \leq $ than either 
$y$ or $z$  
and the same holds for $ cd' $,
hence $a b' \mathrel{  \Psi} cd' $.
We have proved that $\Psi$ is a tolerance
on $ {\mathbf A}$,

Let  $\Theta =\gamma \Psi $
be the intersection of $\gamma$ and $\Psi$ .  
 Then 
 $(   x,   z) \in  \alpha (  \beta \circ  \Theta )$, as witnessed 
by $  y$.
The only (other) element of $  A$ which is  $  \alpha   \beta $-related to 
$  x$ is $ x_1=(1,1,0,0,0) =    x (   y+  z)$ and the only
(other)  element  
$  \alpha   \beta $-related to 
$  z$ is $z_1 = (0,1,0,0,1) =    z (   y'+  x)$.
No non trivial $ \Theta$ relation
holds among the above elements,
besides 
$x \Theta x_1$ and  $z \Theta z_1$ and the converses,  hence 
$(   x,   z) \not\in  \alpha   \beta \circ 
  \alpha    \Theta \circ  \alpha  \beta $ and this shows that
the identity \eqref{id} fails in $ {\mathbf A}$.
See \cite{ia} for further details, comments and variations
related to the present example. 
\end{example}   

Mitschke \cite{Mn}  showed that
$\mathcal I$ 
has no near-unanimity term.
It is thus natural to ask whether there 
exists a congruence $3$-permutable,
$3$-distributive not $2$-distributive variety 
with a near-unanimity term. 
We show that some expansion of 
$\mathcal {I}$ has these properties.

Let 
$u$ be the lattice term defined by
$u(x_1, x_2, x_3, x_4) = \prod _{j \neq j } (x_i + x_j) $,
where the indices on the product vary on the set
$\{ 1,2,3,4\}$. 
The term $u$ is clearly a near-unanimity term 
in every lattice, in particular, in Boolean algebras.
Let $\mathcal I^{nu}$
denote  the variety generated by polynomial  
 reducts of Boolean algebras 
in which 
both  $i$ and $u$
are taken as basic operations.  
We denote by  $\mathcal I^{nu -}$
the variety in which only 
$f$ and  $u$ are considered. 

\begin{proposition} \labbel{unuo}
Both $\mathcal I^{nu}$ and 
 $\mathcal I^{nu -}$
are  $3$-distributive and congruence $3$-permutable
 varieties with a $4$-ary near-unanimity term. 
The varieties $\mathcal I^{nu}$ 
and  $\mathcal I^{nu -}$ are neither 
congruence permutable, nor
$2$-distributive. 
 \end{proposition}

\begin{proof}
Since any congruence permutable and  distributive 
variety is arithmetical, hence $2$-distributive,
it is enough to show that 
$\mathcal I^{nu}$ is not $2$-distributive.
All the rest follows from previous remarks 
and   the mentioned results from \cite{M}.

 Essentially, we are  going to show that the counterexample
to $2$-distributivity presented in Mitschke \cite[p.\ 185]{M}---and
 credited in that form
to the referee---is closed under $u$. Notice that 
we are working with the dual. 
The subset $B = 2^3 \setminus \{ (1,1,1) \} $
of the $8$-elements Boolean algebra    
$\mathbf 2^3$ is clearly closed under $i$.
It is also closed under $u$,
since if   
$b_1, b_2, b_3, b_4 \in B$, then at least two 
$b_j$'s have a $0$ in the same position, hence
$u(b_1, b_2, b_3, b_4 )$ has $0$ in the same position.
If $\alpha$, $\beta$ and $\gamma$ are, respectively,
the kernels of the $2$nd, $1$st and  $3$rd projections, 
then $((1,1,0), (0,1,1)) \in \alpha ( \beta \circ \gamma )$, as witnessed
by the element $(1,0, 1)$. But 
 $((1,1,0), (0,1,1)) \not \in \alpha  \beta \circ \alpha  \gamma $ in $B$,
since the only element which could do the job is
$(1,1,1)$ and 
$(1,1,1) \not \in B$.
\end{proof}

\begin{remark} \labbel{contol}
In \cite{contol}
we showed that, under a fairly general hypothesis,
a variety $\mathcal {V}$ satisfies an identity for congruences if and only if  
$\mathcal {V}$ satisfies the same identity for tolerances,
provided that only tolerances
representable as $R \circ R^\smallsmile $ are taken into account.
By \cite{M}, $\mathcal {I}$ is $3$-distributive 
that is,    $\mathcal {I}$ satisfies the congruence identity
$\alpha( \beta \circ \gamma ) \subseteq \alpha \beta \circ 
\alpha \gamma \circ \alpha \beta  $. On the other hand, 
$\mathcal {I}$ fails to satisfy this identity 
when  $\gamma$ is interpreted as a
tolerance, by Example \ref{anuo}. Alternatively, use Theorem \ref{tol2} and 
 again \cite{M}, where implication algebras are shown not 
to be $2$-distributive.
Hence the assumption of representability is necessary 
in \cite{contol}, even in the case of $3$-distributive congruence $3$-permutable varieties.   
A similar 
counterexample 
for $4$-distributive varieties has been presented   in \cite{baker}.

Notice that the results from 
\cite{contol} imply that every $3$-distributive 
variety satisfies, for example,
 $\alpha( \beta \circ R \circ R ^\smallsmile  ) 
\subseteq \alpha \beta \circ 
\alpha ( R \circ R ^\smallsmile) \circ \alpha \beta  $,
a  result  formally incomparable with the identity
\eqref{3b}.
 \end{remark}

{\scriptsize This is a preliminary version, it might contain inaccuraccies (to be
precise, it is more likely to contain inaccuracies than planned subsequent versions).

The author considers that it is highly  inappropriate, 
 and strongly discourages, the use of indicators extracted from the list below
  (even in aggregate forms in combination with similar lists)
  in decisions about individuals (job opportunities, career progressions etc.), 
 attributions of funds  and selections or evaluations of research projects. \par }

\end{document}